\documentclass{amsart}
\usepackage{latexsym,amsmath,amsfonts,eucal,amscd,amssymb,mathrsfs, graphpap,mathdots,comment}
\usepackage[all]{xy}
\usepackage{graphics}

\DeclareMathOperator{\spec}{Spec}  
  \DeclareMathOperator{\proj}{Proj}

\def\rr{\mathbb{R}}

\def\zz{\mathbb{Z}}
\def\qq{\mathbb{Q}}
\def\pp{\mathbb{P}}

\def\scr#1{\mathscr{#1}}

\def\too{\longrightarrow}

\def\incl{\hookrightarrow}

\def\blow#1{X^{(#1)}}

\def\gcdpv{\mathrm{gcd}_v^+}

\newtheorem{theorem}{Theorem}
\newtheorem{lemma}[theorem]{Lemma}

\newtheorem{conjecture}[theorem]{Conjecture}
\newtheorem{proposition}[theorem]{Proposition}
\theoremstyle{definition}
\newtheorem{definition}[theorem]{Definition}
\newtheorem{remark}[theorem]{Remark}

\newenvironment{thmbis}[1]
  {\renewcommand{\thetheorem}{\ref{#1}$'$}%
   \addtocounter{theorem}{-1}%
   \begin{theorem}}
  {\end{theorem}}

\def\vplus{v^{+}}

\def\blow#1{X_{#1}}

\def\rad{\mathrm{rad}}
\def\trunloc{\lambda_v^{(1)}}

\begin{document}

\title[Vojta's Conjecture on Multiple Blowups and the $abc$]{Vojta's Conjecture on Multiple Blowups of $\pp^2$ and the $abc$ conjecture}
\author{Yu Yasufuku}

\thanks{Supported in part by JSPS Grant-in-Aid 15K17522 and by Nihon University College of Science and Technology Grant-in-Aid for Fundamental Science Research.}

\maketitle

\begin{abstract}
We show that Vojta's conjecture for some rational surfaces is related to the $abc$ conjecture.  More specifically, we prove that Vojta's conjecture on these surfaces implies a special case of the $abc$ conjecture, while the $abc$ conjecture implies Vojta's conjecture on these surfaces.  Moreover, for similar but different rational surfaces, we prove Vojta's conjecture unconditionally.  To prove these results, we use some (possibly new) properties of Farey sequences.

\vspace{0.1in} \noindent Mathematics Subject Classification (2010): 11J97, 14G25, 11J87, 14G40, 11B57, 14J26

\vspace{0.1in} \noindent \textit{Key Words:} Vojta's conjecture, rational surfaces, blowups, subspace theorem,
$abc$ conjecture, Farey sequences.
\end{abstract}

\section{Introduction}

Vojta's conjecture \cite[Main Conjecture (Conjecture 3.4.3)]{vojta} is a deep conjecture in Diophantine geometry which describes how the geometry of the variety controls the distribution of rational points.  It was developed as an analogy with the Griffiths conjecture, which is a high-dimensional version of the Second Main Theorem in Nevanlinna theory.
Vojta's conjecture is stated for a normal-crossing divisor on a smooth variety, and it is extremely deep. Its special cases include the Mordell's conjecture proved by Faltings, the description of integral points on abelian varieties \cite{falabel}, and Schmidt's subspace theorem.  It is also known to imply several conjectures, including Lang's conjecture on degeneracy of rational points on varieties of general type and the Masser-Oesterl\'e $abc$ conjecture, whose proof has been announced by Mochizuki \cite{mochizuki}.

In this paper, we unconditionally prove some special cases of Vojta's conjecture, and also explore the relationship between Vojta's conjecture on certain rational surfaces and the $abc$ conjecture.  The $abc$ conjecture is really ``the'' true analog of the Second Main Theorem of the Nevanlinna theory, since both feature the truncated counting function. In particular, the $abc$ conjecture is considered to be a one-dimensional phenomenon.  The results of this paper shows that the $abc$ conjecture has important consequences and interrelations with Vojta's conjecture on surfaces.

More specifically, we prove three theorems in this paper.  We will defer until the next section the precise statements of Vojta's conjecture and the $abc$ conjecture as well as definitions of local and global heights.  Our first theorem is a completely unconditional proof of Vojta's conjecture in certain case.  In the following, $\sim$ above divisors indicate their strict transforms.

\begin{theorem}\label{thm:vojtawithy}
Let $L_1, L_2, L_3$ be three lines of $\pp^2$ defined over $\overline \qq$ in general position.  We let $X_1$ be the blowup of $\pp^2$ at a point defined over $\overline \qq$ in $L_1 \setminus (L_2 \cup L_3)$ with $E_1$ as the exceptional divisor.  For $n\ge 2$, we construct $X_n$ inductively by blowing up $X_{n-1}$ at (the unique) point of $E_{n-1} \cap \widetilde{L_1}$ and obtaining the exceptional divisor $E_n$.  Then Vojta's conjecture holds for $X_n$ with respect to the divisor
\[
\widetilde{L_1}+\widetilde{L_2}+\widetilde{L_3}+\widetilde{E_1} + \cdots +\widetilde{E_{n-1}} + E_n.
\]
\end{theorem}

The special case of $X_1$ had been treated by our earlier work \cite{yasufuku2}, but this theorem extends to arbitrary number of blowups, as long as the points which are blown up are always on $\widetilde{L_1}$.

We now consider the case of multiple blowups where we start from the same $X_1$, but we blow up at a point not on $\widetilde{L_1}$ at least once.  In the following result, $|x|_S'$ for $x\in \zz$ and a finite set $S$ of primes denotes the prime-to-$S$ part of $x$, namely the largest divisor of $x$ which is relatively prime to any primes in $S$; this notion will be extended to a general element of $\overline \qq$ in the next section.

\begin{theorem}\label{thm:vctoabc}
Let $L_1, L_2, L_3$ be three lines of $\pp^2$ defined over $\overline \qq$ in general position.  We let $X_1$ be the blowup of $\pp^2$ at a point defined over $\overline \qq$ in $L_1 \setminus (L_2 \cup L_3)$ with $E_1$ as the exceptional divisor.  For $n\ge 2$, we construct $X_n$ inductively by blowing up $X_{n-1}$ at a point which is the intersection point of two of $\widetilde{L_1}, \widetilde{E_1}, \ldots, \widetilde{E_{n-2}}, E_{n-1}$, obtaining the new exceptional divisor $E_n$.  Further, let us assume that at least one blowup occurs at a point outside of $\widetilde{L_1}$. Then Vojta's conjecture for such an $X_n$ with respect to the divisor
\[
\widetilde{L_1}+\widetilde{L_2}+\widetilde{L_3}+\widetilde{E_1} + \cdots +\widetilde{E_{n-1}} + E_n
\]
implies the $abc$ conjecture for subsets of the following special form: for a number field $k$, a finite set $S$ of places,
\[
\{(a,b,c): a\in k,  |a|_S' = 1, b=1-a, c=1\}.
\]
\end{theorem}

An $a$ satisfying $|a|_S' = 1$ is called an $S$-unit.  Since every algebraic number becomes an $S$-unit for a large enough $S$, Theorem \ref{thm:vctoabc} says that the $abc$ conjecture is in a sense the ``uniform limit'' of Vojta's conjecture on these surfaces as $S$ is enlarged to the set of all places of $k$.  We will prove a more precise version of this theorem in Section \ref{sec:proof} (see Theorem \ref{thm:vctoabc}$'$ and Remark \ref{rem:vctoabc_ratl}).

We now discuss the implication in the other direction:

\begin{theorem}\label{thm:abctovc}
Let $L_1, L_2, L_3$ be three lines of $\pp^2$ defined over $\overline \qq$ in general position.  We let $X_1$ be the blowup of $\pp^2$ at a point defined over $\overline \qq$ in $L_1 \setminus (L_2 \cup L_3)$ with $E_1$ as the exceptional divisor.  For $n\ge 2$, we construct $X_n$ inductively by blowing up $X_{n-1}$ at a point which is the intersection point of two of $\widetilde{L_1}, \widetilde{E_1}, \ldots, \widetilde{E_{n-2}}, E_{n-1}$, obtaining the new exceptional divisor $E_n$.  Then the $abc$ conjecture implies Vojta's conjecture for $X_n$ with respect to
\[
\widetilde{L_1}+\widetilde{L_2}+\widetilde{L_3}+\widetilde{E_1} + \cdots +\widetilde{E_{n-1}} + E_n.
\]
\end{theorem}

Therefore, Theorems \ref{thm:vctoabc} and \ref{thm:abctovc} show an inter-relation between the $abc$ conjecture and Vojta's conjecture on certain rational surfaces, although the $abc$ conjecture is generally considered to be a Diophantine result in dimension $1$.  It would be interesting to explore why these particular divisors behave this way.  Moreover, since the Diophantine content used in proving Theorem \ref{thm:abctovc} is believed to be significantly deeper than that for Theorem \ref{thm:vojtawithy}, it would be worthwhile to investigate if the two cases can be distinguished geometrically, such as by the intersection matrices of the irreducible components of the divisor. The variety itself is  the $n$-time blowup of $\pp^2$ in both cases, so the key must be in the arrangements of divisors.  These investigations should help prove other relationships between the $abc$ conjecture and Vojta's conjecture on higher-dimensional varieties.

The number-theoretic content of Vojta's conjecture on these surfaces is an inequality of greatest common divisors.  A typical consequence of Theorems \ref{thm:vojtawithy} and \ref{thm:abctovc} as well as their generalization (Proposition \ref{prop:extend}) is the following: given $\epsilon>0$, there exist a constant $C$ and a finite union $Z_\epsilon$ of algebraic curves such that
\[
\gcd(a-1, b) + \sum_{p\notin S} \sum_{i} \mathrm{gcd}_p^+\left(\frac{(a-1)^{b_{i}}}{b^{a_{i}}}, \frac {b^{c_{i}}}{(a-1)^{d_{i}}}\right) < \epsilon \log \max(|a|, |b|) + \log |ab|_S'  + C
\]
holds for all $(a,b)\in \zz^2\setminus Z_\epsilon$, where $\gcd$ is the logarithm of the greatest common divisor, $\mathrm{gcd}_p^+$ is the log gcd of the $p$-adic part of the numerators, and $\frac{a_i}{b_i}$ and $\frac{c_i}{d_i}$ are Farey neighbors.  All of these notions will be defined precisely later.  The connection between Vojta's conjecture on blowups and gcd inequalities is discussed in \cite{sil}, and the results of this article give concrete new examples of this connection.

The proofs of our theorems involve Schmidt's subspace theorem and some geometric computations.  One key turns out to be a property (Theorem \ref{thm:farey}) of Farey fractions.  To prove Theorem \ref{thm:abctovc}, instead of directly computing the $v$-adic contribution in the inequality of Vojta's conjecture, our idea is to rearrange the order of the blowups in a suitable way for each rational point $P$ and each place $v$ so that we can compare with the sequence of blowups that gives the largest $v$-adic contribution for the pair $(P, v)$. This largest contribution will then be computed by utilizing Farey fractions.

We briefly note why we insist of blowing up at intersections.  When we keep the same $X_1$ and then blow up at a point on $E_1$ that is not on $\widetilde{L_1}$, then \cite[Proposition 2]{yasufuku2} shows that Vojta's conjecture on this surface with respect to the divisor
\[
\widetilde{L_1}+\widetilde{L_2}+\widetilde{L_3}+\widetilde{E_1} + E_2
\]
implies the same special case of the $abc$ conjecture as Theorem \ref{thm:vctoabc} above.  Since Vojta's conjecture on blowups is stronger than that on the base, it follows immediately that any further blowups also have this property.  On the other hand, if we start from a different $X_1$, namely if we blow up at a point not on $L_1\cup L_2\cup L_3$, then Vojta's conjecture is not known.  Vojta's inequality on this blowup with respect to $\widetilde{L_1}+\widetilde{L_2}+\widetilde{L_3}+E_1$ has been proved by Corvaja and Zannier \cite{corzan} when restricted to the integral points on $\pp^2\setminus (L_1\cup L_2\cup L_3)$.  This special case already requires quite an ingenious application of Schmidt's subspace theorem, and verifying the conjecture in general in this case is presumably quite difficult.

It is of course natural to consider the counterparts of Theorems \ref{thm:vojtawithy} and \ref{thm:abctovc} in Nevanlinna theory.  As far as we know, these are still open problems, despite the fact that the counterparts of Schmidt's subspace theorem and the $abc$ conjecture (Cartan's theorem and the Second Main Theorem, respectively) are theorems for value-distribution theory.

The paper is organized as follows.  The next section recalls local and global heights, and precisely discusses Vojta's conjecture and the $abc$ conjecture. Section 3 contains the key result from Farey fractions, or more precisely those organized in the so-called Stern--Brocot tree.  Section \ref{sec:geom} contains some geometric computations, and the last section proves Theorems \ref{thm:vojtawithy}--\ref{thm:abctovc}, as well as stating and proving the generalization of Theorems \ref{thm:vojtawithy} and \ref{thm:abctovc} to the situation of starting from blowing up several points of one of the lines of $\pp^2$ (Proposition \ref{prop:extend}).

\section{Background on Heights, and $abc$ and Vojta's Conjecture}
Here we recall local and global heights, and state both Vojta's conjecture and $abc$ conjecture precisely.  We will also state Schmidt's subspace and its consequence which will be used in the proofs of the theorems.

Let $k$ be a number field and $M_k$ be the set of places of $k$.  For $k=\qq$, we have the usual absolute value $|\cdot |_\infty$ and the $p$-adic absolute value, normalized by  $|p|_p = \frac 1p$.  For each $v\in M_k$, we refer to an absolute value as \textbf{normalized} if it is in the equivalence class of $v$ and it restricts to a normalized absolute value on $\qq$.  We then define $|\cdot |_v$ to be the $\frac{[k_v:\qq_v]}{[k:\qq]}$-th power of the $v$-adic normalized absolute value.  We note that $|\cdot |_v$ may not always be an actual absolute value, but only a power of an absolute value.  This should not cause any problems in this paper, as we will not explicitly use triangular inequalities.  With these conventions, each element $x\in k^*$ satisfies the product formula $\prod_{v\in M_k} |x|_v = 1$.  We sometimes use the additive notations: $v(x) = -\log |x|_v$ is the \textbf{normalized additive valuation} for a place and $\vplus(x) = \max(0, v(x))$.

We define the height function on $\pp^n(k)$ by
\[
H([a_0:\cdots : a_n]) = \prod_{v\in M_k} \max(|a_0|_v, \ldots, |a_n|_v)
\]
and $h(P) = \log H(P)$.  This is well-defined because of the product formula, and this is in fact compatible with field extensions so that it is well-defined on $\pp^n(\overline \qq)$.  In particular, we define $H(a) = H([a:1])$ on $\pp^1$; we have $H(\frac ab) = \max(|a|, |b|)$ when $\frac ab\in\qq$ is in a reduced form.  More generally, given a projective variety $X$ defined over $k$ and a Cartier divisor $D$, we can define a Weil height function $h(D, -): X(k) \too \rr$ by $h(\phi_1(P)) - h(\phi_2(P))$, where $D = D_1 - D_2$ with each $D_i$ very ample and $\phi_i:X\incl \pp^{n_i}$ the associated embedding.  This is well-defined up to a bounded function, and in fact a linear equivalence also leaves height functions invariant up to a bounded function.  In addition, the height function satisfies functoriality with respect to pullbacks by morphisms: if $\phi: Y\too X$ is a morphism then $h(\phi^*D, P)$ is the same as $h(D, \phi(P))$ up to a bounded function.

We now introduce local height functions.  There are several ways of defining this, including defining through an integral model and defining through $M_k$-bounded functions.  Here we will not go into details; please refer to \cite{bomgub,serremw}.  Basically, given an effective Cartier divisor $D = \{(U_\alpha, f_\alpha)\}$, the local height function $\lambda_v(D, P)$ is roughly $\max(0, -\log |f_\alpha(P)|_v)$ for $P\in U_\alpha \setminus |D|$.  Therefore, the local height function is big whenever $P$ is $v$-adically close to $D$. To define it precisely, we need to glue these functions appropriately.  As an important example, when $D$ is defined by a homogeneous polynomial $F$ of degree $d$ in $\pp^n$, $\lambda_v(D, [x_0:\cdots :x_n])$ can be given by
\[
v(F(x_0,\ldots, x_n)) - d \min(v(x_0), \ldots, v(x_n)).
\]
Specializing to $\pp^1$, this means that $\lambda_v((0), [a:1]) = \vplus(a)$ for $a\in k^*$, so over $\qq$,  $\lambda_p((0), [a:1])$ is exactly the logarithm of the $p$-adic part of the numerator of $a$.  The local height function also satisfies functoriality with respect to pullbacks by morphisms, and it also satisfies the decomposition property:
\begin{equation}\label{eq:decomp}
h(D, P) = \sum_{v\in M_k} \lambda_v(D,P) + O(1).
\end{equation}
When $a\in k^*$ and $S\subset M_k$ is a finite set of places including all archimedean ones, we can now define the prime-to-$S$ part as
\[
|a|_S' = \exp\left(\sum_{v\notin S} \lambda_v((0), [a:1])\right).
\]
By the definition of local heights, this captures the prime-to-$S$ part of the numerator of $a$.  We also let $\mathcal O_S$ denote the set of $S$-integers of $k$, that is, the set of elements of $k$ whose $v$-adic absolute value is less than or equal to $1$ for all $v\notin S$.  Using our conventions for the local heights, we can also write this as
\[
\mathcal O_S = \{a\in k : \lambda_v((\infty),[a:1]) = 0 \quad \text{for all } v\notin S\}.
\]

For a non-archimedean place $v$, we can also define a truncated local height function.  To do this, let $\gamma_v$ be the minimum strictly-positive value of $v(x)$ as $x$ moves in $k^*$; when $k=\qq$, this is $\log p$ for the $p$-adic place.  We then define the \textbf{truncated local height function} (truncated at $1$) for an effective divisor $D$ by
\[
\trunloc(D, P) = \min\left(\lambda_v(D,P), \gamma_v\right).
\]
This only counts the minimum $v$-adic contribution of local height, even when $P$ is $v$-adically very close to $D$.  We can then define the \textbf{radical} of $a\in k^*$ by
\[
\rad(a) = \exp\left(\sum_{\substack{v\in M_k\\v \text{ non-arch.}}} \trunloc((0), [a:1])\right)
\]
Since this is defined using local heights, this is a real number and not necessarily an element of $k$.  It is clear that when $a\in \zz$, this agrees with the usual notion of the radical, which is the product of primes that divide $a$.   We are now ready to state the $abc$ conjecture. We first state it in the most well-known form.

\begin{conjecture}[$abc$ conjecture]\label{conj:abcQ}
Given $\epsilon>0$, there exists a constant $C>0$ such that whenever $a+b = c$ with $a,b,c\in\zz$ and $\gcd(a,b,c) = 1$, we have
\[
\rad(abc) > C \max(|a|,|b|,|c|)^{1-\epsilon}.
\]
\end{conjecture}

This says that whenever integers satisfy an additive relation, one cannot expect a nice multiplicative structure that they are all divisible by high powers of small primes.  Since we have introduced the notion of truncated local height functions, it is easy to state the number-field version:

\begin{conjecture}[$abc$ conjecture for number fields]\label{conj:abc}
Let $k$ be a number field and let $S$ be a finite set of places including all archimedean ones.  For $\epsilon >0$, there exists a constant $C$ such that for all $x\in \pp^1(k)$
\begin{equation}\label{ineq:abc}
\sum_{v\notin S} \left(\trunloc((0), x)+\trunloc((1), x)+\trunloc((\infty), x)\right) \ge (1-\epsilon) h(x) \,\,+\,\, C.
\end{equation}
\end{conjecture}

The proof of this has been announced by Mochizuki \cite{mochizuki}.   By applying this conjecture for $k=\qq$ and $x = \frac ac$, it is clear from the definitions that we obtain Conjecture \ref{conj:abcQ}.

We now state Vojta's conjecture \cite[Main Conjecture (Conjecture 3.4.3)]{vojta}:

\begin{conjecture}[Vojta's Conjecture]\label{conj:vojta}
Let $X$ be a smooth projective variety over $k$, $K$ a canonical divisor of $X$, $A$ an ample divisor and $D$ a normal-crossings divisor. Fix height functions $\lambda_v(D, -)$, $h(K,-)$, and $h(A,-)$.  Let $S\subset M_k$ be a finite set of places.  Then given $\epsilon>0$, there exist a Zariski-closed $Z_\epsilon\subsetneq X$ and a constant $C$ such that
\begin{equation}\label{eq:vojtacon}
\sum_{v\in S} \lambda_v(D, P) + h(K, P) < \epsilon h(A, P) + C
\end{equation}
for all $P\in X(k)$ not on $Z_\epsilon$.
\end{conjecture}

Here, normal-crossings divisors are assumed to be \textit{reduced} by definition.  Since local height functions for $S$ are big whenever the point is close to the divisor for places in $S$, \eqref{eq:vojtacon} says that rational points can get $v$-adically close to $D$  for $v\in S$ as much as there is negativity in $K$.  In this sense, this conjecture is a quantitative statement about how the geometry of the canonical divisor controls the arithmetic of existence of rational points close to a divisor.

Vojta's conjecture for $\pp^1$ is equivalent to the multi-place generalization of Roth's theorem.
Moreover, rearranging terms in \eqref{eq:vojtacon} for $D = (0) + (1) + (\infty)$ on $\pp^1$, we obtain
\[
(1-\epsilon) h(P) < \sum_{v\notin S} \left(\lambda_v((0), P) + \lambda_v((1), P) + \lambda_v((\infty), P)\right) \,\, + C'.
\]
We now see that the stronger form of Roth's theorem with the \textit{truncated} local functions in place of the usual local functions is precisely the $abc$ conjecture (Conjecture \ref{conj:abc}). It is a highly nontrivial result of Vojta \cite{vojtaabc} that the $abc$ conjecture follows from Vojta's conjecture (Conjecture \ref{conj:vojta}) on an arbitrarily-high dimensional varieties.

One known case of Vojta's conjecture is Schmidt's subspace theorem.  This is the case of $X = \pp^n$ and $D$ being the normal-crossings union of hyperplanes. We will present this theorem in the following form (cf. \cite[Theorem 7.2.2]{bomgub}):

\begin{theorem}[Schmidt's Subspace Theorem]\label{thm:schmidt}
Let $S\subset M_k$ be a finite set of places of $k$.  For each $v\in S$, assume that $L_{v,0},\ldots, L_{v,n}$ are linearly independent linear
forms in $(n+1)$ variables with coefficients in $k$.  Then given $\epsilon>0$, any $x = (x_0, \ldots, x_n)\in k^{n+1}\setminus \{(0,\ldots, 0)\}$ with
\[
\prod_{v\in S}\prod_{i=0}^n \frac{|L_{v,i}(x)|_v}{\max(|x_0|_v, \ldots, |x_n|_v)} < H(x)^{-n-1-\epsilon}
\]
is contained in finitely many linear subspaces.
\end{theorem}

We end this section by mentioning the following result. This is actually an immediate consequence of Theorem \ref{thm:schmidt} for $n=1$, but the exact statement and its proof do not seem to be in the literature, so we include the short proof for convenience.

\begin{proposition}\label{prop:ridout}
Let $S$ be a finite set of places of $k$.  Given $\epsilon>0$, there exists a constant $C$ such that
\[
\sum_{v\in S} \vplus(a-1) < \epsilon h(a) + \sum_{v\notin S} |v(a)| + C
\]
for all $a\in k\setminus \{0, 1\}$.
\end{proposition}

In particular, when $a\in \zz$, the $S$-part of $a-1$ is roughly at most the prime-to-$S$ part of $a$.  Moreover, even when $a$ has some denominators, this says that the $S$-part of the numerator of $a-1$ is roughly at worst  the prime-to-$S$ part of the product of the numerator and the denominator of $a$.

\begin{proof}
Let $T$ be a subset of $S$.  We apply Schmidt's subspace theorem for two variables (so really a multi-place version of Roth's theorem) with linear forms $X-Y$ and $Y$ for $v\in T$ and linear forms $X$ and $Y$  for $v\in S\setminus T$.  We then plug in $(\frac 1a, 1)$:
\[
\frac{\displaystyle \left(\prod_{v\in T} |a-1|_v\right)\times \left(\prod_{v\in S} \left| \frac 1a \right|_v\right)}{\displaystyle \prod_{v\in S} \max\left(\left| \frac 1a \right|_v, 1 \right)^2} > H(a)^{-2-\epsilon}.
\]
Since $H(a) = H(1/a) = \prod_{v\in M_k} \max(|1/a|_v, 1)$, it follows from the product formula that
\[
\prod_{v\in T} |a-1|_v > \left(\prod_{v\notin S} \left| \frac 1a \right|_v \right) \times \left( \prod_{v\notin S} \max\left(\left| \frac 1a\right|_v, \,\,  1\right)^{-2}\right) \times H(a)^{-\epsilon}
\]
For each $v\notin S$, the contribution from the right-hand side is $|a|_v$ when $|a|_v \le 1$ and is $1/|a|_v$ when $|a|_v > 1$.  We have finitely many exceptions for each $T$, so the total number of exceptions is still finite as $T$ moves among all subsets of $S$.  By using $T = \{v\in S: |a-1|_v < 1\}$ for each $a$ and taking $-\log$ of both sides, we obtain the proposition, accommodating all the exceptions by adjusting the constant $C$.
\end{proof}

\section{Stern--Brocot Tree}\label{sec:farey}

In this section, we first recall the definition of the Stern--Brocot tree and its basic properties.  We then prove a property which will be crucial in our proof of Vojta's conjecture on rational surfaces.

\begin{definition}
The \textbf{mediant} of two fractions $\frac ab$ and $\frac cd$ is $\frac{a+b}{c+d}$.  The (left-side of) the \textbf{Stern--Brocot tree} is constructed as follows: we begin with $\frac 01$ and $\frac 11$ as the level $1$ elements of the tree; we then inductively construct elements of level $k+1$ by inserting the mediants of every neighboring elements of level $k$.  The fractions in the first four levels of the tree are depicted in the picture on the next page ($\bullet$ and the line segments will be explained later).
\end{definition}

\begin{figure}[h]
\[
\xymatrix@=1pt{\frac 01 &&&&&&&&&&&&&&&& \bullet \ar@{-}[lllllllldddd]\ar@{-}[rrrrrrrrdddd] &&&&&&&&&&&&&&&& \frac 11 &&&\text{level } 1\\
& \\ & \\ & \\
\frac 01 &&&&&&&& \bullet \ar@{-}[lllldddd]\ar@{-}[rrrrdddd] &&&&&&&& \frac 12 &&&&&&&& \bullet
\ar@{-}[lllldddd]\ar@{-}[rrrrdddd] &&&&&&&& \frac 11
&&&\text{level }2\\
& \\ & \\ & \\
 \frac 01 &&&& \bullet \ar@{-}[lldddd]\ar@{-}[rrdddd] &&&& \frac 13 &&&& \bullet \ar@{-}[lldddd]\ar@{-}[rrdddd] &&&&
\frac 12 &&&& \bullet \ar@{-}[lldddd]\ar@{-}[rrdddd] &&&& \frac 23 &&&& \bullet \ar@{-}[lldddd]\ar@{-}[rrdddd] &&&& \frac 11 &&&\text{level }3\\
& \\ & \\ & \\
\frac 01 && \bullet \ar@{-}[ldddd]\ar@{-}[rdddd] && \frac 14 && \bullet \ar@{-}[ldddd]\ar@{-}[rdddd] && \frac 13 &&
\bullet \ar@{-}[ldddd]\ar@{-}[rdddd] && \frac 25 && \bullet \ar@{-}[ldddd]\ar@{-}[rdddd] && \frac 12 && \bullet
\ar@{-}[ldddd]\ar@{-}[rdddd] && \frac 35 && \bullet \ar@{-}[ldddd]\ar@{-}[rdddd] &&
\frac 23 && \bullet \ar@{-}[ldddd]\ar@{-}[rdddd] && \frac 34 && \bullet \ar@{-}[ldddd]\ar@{-}[rdddd] && \frac 11 &&&\text{level }4\\
& \\ & \\ & \\
 &&&&&&&&&&&&&&&&&&&&&&&&&&&&&&&&&&&&&&   }
\]
\caption{Stern--Brocot Tree}
\label{fig:sbtree}
\end{figure}
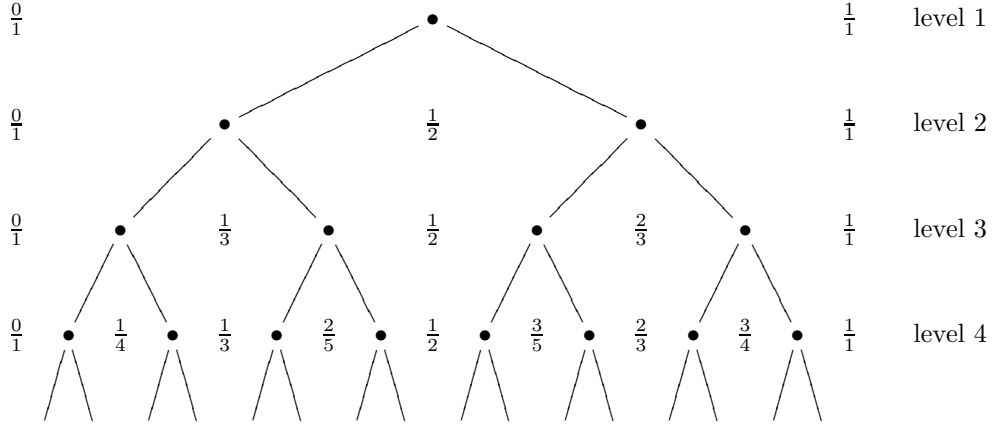

The usual Stern--Brocot tree actually starts with $\frac 11$ serving as the root in the middle of $\frac 01$ and $\frac 10$, and proceeds by inserting mediants of neighbors at each level.  In this paper, we will only need the fractions less than or equal to $1$.  The Stern--Brocot tree often appears as an efficient way of constructing Farey sequences, which have been important in various fields, including binary search algorithms and the circle method in analytic number theory.

We now quickly list several properties of the Stern--Brocot tree; for more details see for example \cite{graham_knuth}. Every new mediant is a reduced fraction.  At level $n$, the interval $[0,1]$ is subdivided into $2^{n-1}$ subintervals, so we have $2^{n-1} + 1$ fractions in level $n$.  Every reduced fraction $\frac ab$ between $0$ and $1$ appears by level $b$.  If $\frac ab$ and $\frac cd$ are neighbors at some level, then we have $|bc - ad | = 1$.

Before proving the main result of this section, we will introduce the following notation.

\begin{definition}
Given a natural number $n$ and $x\in [0,1]$, we define the \textbf{Farey interval} of $x$ at level $n$ to be the subinterval of $[0,1]$ in level $n$ of the Stern--Brocot tree that contains $x$.  We let $\mathbf{a_n(x)}$, $\mathbf{b_n(x)}$, $\mathbf{c_n(x)}$, and $\mathbf{d_n(x)}$  denote respectively the numerator of the initial point, the denominator of the initial point, the numerator of the end point, and the denominator of the end point of the Farey interval of $x$ at level $n$.  We do not define any of these notions when $x$ itself appears at level $n$ of the tree.
\end{definition}

We are now ready to prove the following key result.

\begin{theorem}\label{thm:farey}
Suppose that $\alpha \in \qq \cap (0,1)$ appears in the Stern--Brocot tree for the first time in level $n+1$ for some natural number $n$.  Then (the closure of) the graph of the real-valued function $\varphi_\alpha$ on $I_\alpha = \displaystyle \left(\frac{a_n(\alpha)}{b_n(\alpha)}, \,\, \frac{c_n(\alpha)}{d_n(\alpha)}\right)$ defined by
\begin{equation}\label{eq:farey_func}
\varphi_\alpha(x) = \sum_{i=1}^n \min \left(b_i(\alpha)\cdot  x - a_i(\alpha), \,\, c_i(\alpha) - d_i(\alpha) \cdot x\right)
\end{equation}
consists of two line segments connecting
\begin{align*}
&\left(\frac{a_n(\alpha)}{b_n(\alpha)}, \frac{b_n(\alpha) - 1}{b_n(\alpha)}\right), \,\, \left(\frac{a_n(\alpha) + c_n(\alpha)}{b_n(\alpha) + d_n(\alpha)}, \frac{b_n(\alpha) + d_n(\alpha) - 1}{b_n(\alpha) + d_n(\alpha)}\right),  \,\,\left(\frac{c_n(\alpha)}{d_n(\alpha)}, \frac{d_n(\alpha) - 1}{d_n(\alpha)}\right)
\end{align*}
in this order.  In particular, the maximum of $\varphi_\alpha$ on $I_\alpha$ occurs at $\alpha = \frac{a_n(\alpha)+c_n(\alpha)}{b_n(\alpha) + d_n(\alpha)}$.
\end{theorem}

We note that the functions $a_i$, $b_i$, $c_i$, and $d_i$ are all constant on $I_\alpha$.  Therefore, we can also define $\varphi_\alpha(x)$ by replacing all appearances of $\alpha$ on the right-hand side of \eqref{eq:farey_func} by $x$.  Since $b_n(\alpha) + d_n(\alpha)$ is the denominator of $\alpha$, it follows from this theorem that the maximum value of $\varphi_\alpha$ on $I_\alpha$ is $\frac{\mathrm{denom}(\alpha)-1}{\mathrm{denom}(\alpha)}$.

\begin{proof}
We prove this by induction.   When $n=1$, $\alpha = \frac 12$, and $\varphi_\alpha(x) = \min(x, 1-x)$, so the theorem is obvious.  Now assume that we have proved the statement for $n$, and assume that $\alpha$ appears for the first time in level $n+2$.  To alleviate notations, let $a = a_{n+1}$, and let $b, c, d$ be similarly defined.  In particular, we have $\alpha = \frac{a+ c}{b + d}$.  We note that from the construction of the Stern--Brocot tree, exactly one of $\frac ab$ or $\frac cd$ is also an endpoint at level $n$, while the other is created as a mediant of the endpoints at level $n$.  Therefore, it follows from the inductive hypothesis that the graph of the summation up to $n$ in $\varphi_\alpha$ is a single line segment connecting $(\frac ab, \frac{b-1}b)$ to $(\frac cd, \frac {d-1}d)$.   The contribution to $\varphi_\alpha$ at level $n+1$ is given by
\[
\min\left(b \cdot x - a, \,\, c - d \cdot x\right),
\]
and this is a $\wedge$-shaped function which has the value $0$ at the two end points $\frac{a}{b}$ and $\frac cd$ with the peak at the mediant $\alpha = \frac {a+c}{b+d}$.  Therefore, it immediately follows that $\varphi_\alpha$ again consists of two line segments, the first starting at $(\frac ab, \frac{b-1}b)$ and the second ending at $(\frac cd, \frac {d-1}d)$.  The $y$-coordinate of the point connecting the two segments can be computed as follows, using the fact that $bc-ad = 1$:
\begin{align*}
&\left[\frac{\frac{d-1}d - \frac{b-1}b}{\frac cd - \frac ab} \left(\alpha - \frac ab\right) + \left(1-\frac 1b\right)\right] + \left(b\alpha -a\right)\\
&=(d-b)\left(\alpha - \frac ab\right) + \left(1-\frac 1b\right) + \left(b\alpha -a\right)\\
&=d\left(\alpha - \frac ab\right) + 1-\frac 1b\\
&=d\left(\frac{a+c}{b+d} - \frac ab\right) + 1 - \frac 1b = d\cdot \frac{bc-da}{(b+d)b} + 1 - \frac{b+d}{(b+d)b}\\
&= 1 - \frac 1{b+d}.
\end{align*}
Therefore, the statement also holds for $n+1$ and the induction continues.
\end{proof}

\section{Some Geometric Computations on Multiple Blowups of $\pp^2$}\label{sec:geom}

In this section, we prove some geometric results which will be useful in proving the main theorems of the paper.  We first recall two propositions proved in \cite{yasufuku2}, which we will again use in this article.  The first computes local heights of exceptional divisors, and the second is how the height with respect to an ample divisor on the base is related to the height with respect to an ample divisor on the blowup.

\begin{proposition}\label{prop:blowupht}
Let $X$ be a smooth projective variety over $k$, and let $Q\in X(k)$.  Let $U$ be an open affine
neighborhood of $Q$, and let $x_1,\ldots, x_n \in \scr O_X(U)$ define $Q$.  Then the local height on the blowup $\widetilde X$ of $X$
at $Q$ with respect to the exceptional divisor $E$ can be given by
\[
\lambda_v(E, P) = \max(0, \min_j v(x_j(P)),
\]
where we identify $\widetilde X\setminus E$ with $X\setminus Q$, and where we extend each $x_j$ as a rational function on
$X$, interpreting $x_j(P)$ to be $\infty$ whenever $x_j$ is undefined.
\end{proposition}

\begin{proposition}\label{prop:ample}
Let $X$ be a projective scheme over $k$ with an ample divisor $A$, and let $\scr I$ define a closed subscheme.  Let $\pi: \widetilde X\too X$ be
the blowup of $X$ along $\scr I$. Then there exists a very ample divisor $A'$ on $\widetilde X$ such
that
\[
h_X(A, \pi(P)) \le h_{\widetilde X}(A', P) + O(1)
\]
for all $P\in \widetilde X(k)$.  Moreover, there is a constant $C$ such that
\[
h_{\widetilde X}(A', P) \le C \cdot h_X(A, \pi(P)) \qquad \qquad \text{for all } P\in \widetilde X(k)
\]
\end{proposition}

For a non-archimedean $v$, $\vplus(\alpha)$ measures the logarithm of the $v$-adic part of the numerator of $\alpha$, so $\lambda_v(E,P)$ can be thought of as the logarithm of the $v$-adic part of the GCD of the numerators of $x_1(P), \ldots, x_n(P)$.  Therefore, we will sometimes use the notation $\gcdpv(x_1(P),\ldots, x_n(P))$ in place of $\max(0, \min_j v(x_j(P))$, even for an archimedean $v$.

For the rest of this section, we will be considering the following situation.  Choosing coordinates on $\pp^2$, we let $\Delta$ be the divisor $XYZ = 0$, and let the first blowup occur at the point $[1:0:1]$.  This is on $Y=0$, but not on $XZ = 0$, and let us denote the line $Y=0$ by $L$.   As we continue the blowup process, we denote by $X_i$ the space created after the $i$-th blowup, and $E_i$ the exceptional divisor created at the $i$-th blowup inside $X_i$.  In particular, $E_1$ is the first exceptional divisor created over $[1:0:1]$. To alleviate notations as much as possible, we will denote by $\sim$ above a divisor to indicate the strict transform of the divisor to the space currently under consideration (usually the $n$-th blowup $X_n$).  Since we will not vary the space in question, the notation without specifying $n$ should not cause any confusion.  Moreover, we will simply denote by $\pi$ the blowup morphism from the space currently under consideration (again, this is usually the $n$-th blowup of $\pp^2$) to various earlier stages of the blowups (including $\pp^2$).  To be accurate, one really should denote by $\pi_{n,i}$ for example the blowup morphism from $X_n$ to $X_i$, but we will denote all of these by just $\pi$.  This should not cause any confusions, as it should be clear that $\pi$ in $\pi^*(E_i)$ should  really be $\pi_{n,i}$. We will sometimes use $\pi_n$ to clearly indicate the blowup map from the $n$-th blowup $X_n$ to $\pp^2$.

Since each blowup in our blowup sequence occurs at a (reduced) point inside a smooth surface, it follows from \cite[Exercise II.8.5]{harts} that the canonical divisor acquires one copy of the newest exceptional divisor at every blowup.  Since the canonical divisor of $\pp^2$ can be written as $-\Delta$, it follows immediately that $-\pi_1^*(\Delta) + E_1$ is a canonical divisor on $X_1$.  We then prove by induction that
\begin{equation}\label{eq:canon_div}
-\pi_n^*(\Delta) + \pi^*(E_1) + \cdots + \pi^*(E_{n-1}) + E_n
\end{equation}
is a canonical divisor on $X_n$.

We end this section by computing the normal-crossings divisor on $X_n$ related to $\Delta$, as well as the (local) coordinate ring of $X_n$.  Since heights behave well with respect to pullbacks by morphisms, we will express our divisors in terms of pullbacks from various stages of the blowups.

\begin{proposition}\label{prop:divandcoordring}
Let $X_1$ be the blowup of $\pp^2$ at $[1:0:1]$, let $\Delta = (XYZ = 0)$ and $L = (Y=0)$.  For each $i\ge 1$, let $X_{i+1}$ be the blowup of $X_i$ at one of the intersection points of $\widetilde L$, $\widetilde{E_1}$, $\ldots$, $\widetilde{E_{i-1}}$, and $E_i$.
\begin{itemize}
\item[(i)] For $n\ge 2$, the maximal normal-crossings divisor on $X_n$ which maps to $\Delta$ in $\pp^2$ is
\begin{equation}\label{eq:div_intersect}
D_n = \pi^*(\Delta) - \pi^*(E_2) - \cdots - \pi^*(E_{n-1}) - E_n.
\end{equation}
\item[(ii)] Suppose further that $X_{i+1}$ is the blowup of $X_{i}$ at a point on $E_{i}$ (so we successively blow up at a point on the most-recently-created exceptional divisor).  Then for all $n\ge 1$, $\blow {n+1}$ is obtained from $\blow {n}$ by blowing up along the ideal locally defined by
\[
\left(\frac{(x-1)^{b_n}}{y^{a_n}}, \,\, \frac{y^{c_n}}{(x-1)^{d_n}}\right),
\]
where $\frac{a_n}{b_n}$ and $\frac{c_n}{d_n}$ are Farey neighbors at level $n$.
\end{itemize}
\end{proposition}

Note that for (i), we do not assume that the blowup always takes place at a point on the most recently-constructed exceptional divisor.  (ii) will be used in conjunction with Proposition \ref{prop:blowupht} to compute global and local heights with respect to exceptional divisors in our setting.

\begin{proof}
\noindent (i)
We note that $\pi_1^*(\Delta) = \widetilde \Delta + E_1$, which is a normal-crossings divisor on $X_1$.  We only have one possibility next, namely blowing up at the intersection of $\widetilde L$ and $E_1$.  Therefore, $\pi_2^*(\Delta) = \widetilde \Delta + \widetilde{E_1} + 2 E_2$ on $X_2$, so the maximal normal-crossings divisor on $X_2$ which maps to $\Delta$ on $\pp^2$ is $D_2 = \pi_2^*(\Delta) - E_2$.  This serves as the base case for the induction.  Suppose we have already proved the result for $n$, then the next blowup occurs at some intersection point among the divisors in $D_n$ which is not $\widetilde{(X=0)}$ or $\widetilde{(Z=0)}$. It follows that
\[
\widetilde{D_n} + 2 E_{n+1} = \pi^*(D_n) = \pi^*(\Delta) - \pi^*(E_2) - \cdots - \pi^*(E_{n-1}) - \pi^*(E_n),
\]
and the induction continues with $D_{n+1} = \widetilde{D_n} + E_{n+1}$.

\noindent (ii)
In general, if $\spec A$ is a smooth surface and the ideal $(f, g)$ defines a smooth point $P$, then the blowup at $P$ can be described by $\proj A[\alpha:\beta]/(f\beta - g\alpha)$.  Moreover, if we look at the affine patch where $\alpha\neq 0$, namely $\spec A[\frac \beta \alpha]/(f\frac \beta \alpha - g)$, the equation $\frac \beta \alpha = 0$ defines the strict transform of the curve $g=0$ while $f=0$ defines the exceptional divisor.  Similarly, on the affine patch where $\beta \neq 0$, namely $\spec A[\frac\alpha \beta]/(f - g\frac \alpha\beta)$, the equation $\frac \alpha \beta = 0$ defines the strict transform of the curve $f=0$ while $g=0$ defines the exceptional divisor.  Combined, this means that as we blowup at $P$, the two curves $f=0$ and $g=0$ on $\spec A$ are separated, and their strict transforms meet the exceptional divisor at $[0:1]$ and $[1:0]$ in the above-mentioned choice of coordinates.

Now, we compute the coordinate ring of our blowups.  The projective coordinate ring around the blowup point for the first blowup $\blow 1$ is $\proj k[x,y][\alpha_1:\beta_1]/ ((x-1)\beta_1 - y \alpha_1)$.  Next we blow up $\blow 1$ at $\widetilde{(Y=0)}\cap E_1$.  This point is defined locally by $x-1$ and $\frac{\beta_1}{\alpha_1}$, so by the discussion of the above paragraph, the second blowup locally looks like
\[
\proj \left(k\left[x,y,\frac{\beta_1}{\alpha_1}\right] \big/ \left((x-1)\frac{\beta_1}{\alpha_1} - y\right)\right)[\alpha_2:\beta_2] \,\, \big/ \,\, \left((x-1)\beta_2 - \frac{\beta_1}{\alpha_1} \alpha_2\right).
\]
Note that in this case, assuming that $x-1\neq 0$ and $y\neq 0$, we have $\frac{\beta_1}{\alpha_1} = \frac y{x-1}$.  Therefore, we can also say that $\blow 2$ was obtained by blowing up along the ideal defined locally by
\[
\left(\frac{(x-1)^{b_1}}{y^{a_1}}, \,\, \frac{y^{c_1}}{(x-1)^{d_1}}\right),
\]
where $\frac{a_1}{b_1} = \frac 01$ and $\frac{c_1}{d_1} = \frac 11$ are the end points of the interval on the first level of the Stern--Brocot tree.  This is the base case of the induction.  Let us assume that we have proved the claim for $n$.  Since the divisor $\widetilde{\Delta} + \widetilde{E_1} + \cdots \widetilde{E_{n-1}} + E_{n}$ is normal-crossings on $\blow {n}$, there are at most two divisors going through any given point, including the point we blow up to obtain $\blow {n+1}$.  This means that among $\widetilde{\Delta}, \widetilde{E_1}, \ldots, \widetilde{E_{n}}$, the only divisors that meet $E_{n+1}$ on $\blow {n+1}$ are exactly the ones that are locally defined by $\frac{(x-1)^{b_n}}{y^{a_n}}$ and $\frac{y^{c_n}}{(x-1)^{d_n}}$.  Therefore, by the first paragraph of the proof, the intersection points on $E_{n+1}$ in $\blow {n+1}$ are defined locally by either
\[
\frac{(x-1)^{b_n}}{y^{a_n}} \quad \text{ and } \quad \frac{y^{c_n}}{(x-1)^{d_n}} \big/ \frac{(x-1)^{b_n}}{y^{a_n}}  = \frac{y^{a_n+c_n}}{(x-1)^{b_n+d_n}}
\]
or
\[
\frac{(x-1)^{b_n}}{y^{a_n}}  \big/ \frac{y^{c_n}}{(x-1)^{d_n}}  = \frac{(x-1)^{b_n+d_n}}{y^{a_n+c_n}} \quad \text{ and } \quad  \frac{y^{c_n}}{(x-1)^{d_n}}.
\]
Hence, the induction continues to $n+1$.
\end{proof}

\begin{remark}\label{rem:blowupsbtree}
We can think of each $\bullet$ in the Stern--Brocot tree (Figure \ref{fig:sbtree}) at level $n$ as a possibility for the $(n+1)$-th blowup in the setting of Proposition \ref{prop:divandcoordring} (ii), by viewing two neighboring fractions to each $\bullet$ as representing the powers of $x-1$ and $y$ in the local equations of the ideal along which the most recent blowup occurred.  The first blowup at $[1:0:1]$ of $\pp^2$ is not depicted in the tree, and there is only one possibility at level $1$ of the tree (corresponding to $\blow 2$).  From there, whenever we blow up at $[1:0]$ in the coordinate system described above, we take the left fork to the next level of the tree, and whenever we blow up at $[0:1]$, we take the right fork to the next level.  It therefore follows that continuing to blow up at the (unique) point on $E_n\cap \widetilde{(Y=0)}$ amounts to taking the left fork all the time, so the ideal along which we blow up to obtain $\blow {n+1}$ in this case is defined locally by $x-1$ and $\frac{y}{(x-1)^n}$.  We will make use of this observation in the proof of Theorem \ref{thm:vojtawithy}.
\end{remark}

\section{Proofs of Theorems}\label{sec:proof}
We now prove Theorems \ref{thm:vojtawithy}--\ref{thm:abctovc}.  Since projective equivalences do not change the content of Vojta's conjecture, we may assume without loss of generality that the three lines $\Delta$ are defined by $XYZ = 0$ and our first blowup is at the point $[1:0:1]$, which is on $L_1 = (Y=0)$ but not on the other two lines.  To shorten notation, we will denote $L_1$ just by $L$.  We will continue to use the notations from Section \ref{sec:geom}.

\begin{proof}[Proof of Theorem \ref{thm:vojtawithy}]
For this theorem, we always blow up at the intersection point with $\widetilde L$.  Since two curves meeting transversally get separated by the blowup at the intersection point, it is easy to see by induction that there is only one point on $\widetilde L$ in $X_{n-1}$ which intersects with $\widetilde{E_1} \cup \cdots \cup \widetilde{E_{n-2}} \cup E_{n-1}$, namely the intersection point with the most-recently-created exceptional divisor $E_{n-1}$.  Hence, it follows that there is a unique sequence of blowups which satisfies the hypothesis of Theorem \ref{thm:vojtawithy}.

Since we may put finitely many curves into the exceptional set of Vojta's conjecture, we will only need to confirm the veracity of Vojta's inequality \eqref{eq:vojtacon} for points of the form $P = [a:b:1]$ for $a\in k\setminus \{0,1\}$ and $b\in k^*$.  Note that these points avoid the images under $\pi$ of the exceptional divisors, so they correspond to a unique point in $X_n$; we will often implicitly identify $P$ with the corresponding point in the various stages of the blowup.  The blowup point for the first blowup is defined by $x-1$ and $y$; the blowup point for the second blowup is (locally) defined by $x-1$ and $\frac y{x-1}$, and as noted in Remark \ref{rem:blowupsbtree}, the blowup point on $X_{i-1}$ for the $i$-th blowup is locally defined by $x-1$ and $\frac y{(x-1)^{i-1}}$.  Therefore, using the description \eqref{eq:canon_div} of the canonical divisor $K_n$ on $X_n$ and the description \eqref{eq:div_intersect} of the normal-crossings divisor $D_n$, the Vojta's inequality on $X_n$ with respect to $D_n$ is
\begin{align*}
\sum_{v\in S} &\lambda_v(D_n, P) + h_{X_n}(K_n, P)\\
&= \sum_{v\in S} \lambda_v(\pi^*(\Delta) - \pi^*(E_2) - \cdots - \pi^*(E_{n-1}) - E_n, P) \\
&\phantom{\lambda_v(\pi^*(\Delta) - \pi^*(E_2) - \cdots)} + h_{X_n}(-\pi_n^*(\Delta) + \pi^*(E_1) + \cdots + \pi^*(E_{n-1}) + E_n, P)\\
&= - \sum_{v\notin S} \lambda_v(\Delta, P) + h_{X_1}(E_1, P) + \sum_{v\notin S} \sum_{i=2}^n \lambda_v(E_i, P) \,\,+\,\, O(1)\\
&< \epsilon h_{X_n}(A_n, P) + O(1),
\end{align*}
where $A_n$ is an ample divisor on $X_n$ and we used \eqref{eq:decomp} and the functoriality with respect to pullbacks in the second equality.
Note that we can calculate all the relevant local height functions using Propositions \ref{prop:blowupht} and \ref{prop:divandcoordring} (ii), as well as Remark \ref{rem:blowupsbtree}.  Using \eqref{eq:decomp} and Proposition \ref{prop:ample}, it now suffices to show the following inequality:
\begin{multline}\label{eq:vc_intersect_detail}
\sum_{v\in M_k} \gcdpv(a-1, b)  + \sum_{i=2}^{n} \sum_{v\notin S} \gcdpv\left(a-1, \frac b{(a-1)^{i-1}}\right) \\
< \epsilon\max (h(a), h(b)) + \sum_{v\notin S} \lambda_v((XYZ=0), [a:b:1]) + C.
\end{multline}
We will first prove the following elementary lemma:
\begin{lemma}\label{lem:elementary}
For a real number $\alpha \in [0,1]$ and a natural number $\ell$, we always have
\[
\alpha + \max(0, \min(\alpha, 1-\alpha)) + \max(0, \min(\alpha, 1-2\alpha)) + \cdots + \max(0, \min(\alpha, 1-\ell\alpha)) \le 1.
\]
\end{lemma}
\begin{proof}
When $\alpha =0$ or $1$, the claim is obvious as the left-hand side is exactly $\alpha$.  Otherwise, for each $\alpha$, there exists a unique natural number $m$ such that $\frac 1{m+1} \le \alpha < \frac 1m$.  In this case, $\min(\alpha, 1-i\alpha) = \alpha$ for $0\le i<m$, $\min(\alpha, 1- m\alpha) = 1-m\alpha$ and $\min(\alpha, 1-i\alpha)\le 0$ for $i\ge m+1$.  Therefore, the left-hand side of the desired inequality is
\[
\begin{cases}
(\ell+1) \alpha < 1 &\text{if } \ell < m\\
\underbrace{\alpha+\cdots + \alpha}_{m\text{ times}} + (1-m\alpha) = 1 &\text{if } \ell\ge m
\end{cases}
\]
and thus the lemma has been proved.
\end{proof}

Going back to the proof of the theorem, we first analyze \eqref{eq:vc_intersect_detail} for each $v\notin S$.  When $v(a-1) \le 0$ or $v(b) \le v(a-1)$, the $v$-part of the left-hand side is obviously bounded by $\max(0, v(b)) \le v(b) - \min(v(a), v(b), 0) = \lambda_v((Y = 0), [a:b:1])$.  Otherwise, we use Lemma \ref{lem:elementary} with $\alpha = \frac{v(a-1)}{v(b)}$ to conclude that the $v$-adic part of the left-hand side is bounded by $v(b)$.  Therefore, it follows that
\begin{equation}\label{eq:intersect_outside_S}
\gcdpv(a-1, b) + \sum_{i=2}^{n} \gcdpv\left(a-1, \frac b{(a-1)^{i-1}}\right) \le \lambda_v((Y=0), [a:b:1])
\end{equation}
for each $v\notin S$.  It is important that we can state \eqref{eq:intersect_outside_S} without constants, as there are infinitely many $v\notin S$.  We now work with the $S$-part of \eqref{eq:vc_intersect_detail}.  We note that
\[
\sum_{v\in S} \gcdpv(a-1, b) \le \sum_{v\in S} \max(0, v(a-1))
\]
and this in turn is bounded by
\begin{equation}\label{eq:intersect_inside_S}
\epsilon h(a) + \sum_{v\notin S} |v(a)| + C
\end{equation}
by Proposition \ref{prop:ridout}.  Now, note that $\lambda_v((X=0), [a:b:1]) \ge v(a)$ when $v(a)$ is positive and $\lambda_v((Z=0), [a:b:1]) \ge -v(a)$ when $v(a)$ is negative.  Therefore, \eqref{eq:intersect_inside_S} is in turn less than or equal to
\[
\epsilon \max(h(a), h(b)) + \sum_{v\notin S} \lambda_v((XZ =0), [a:b:1]) + C.
\]
Combining this relation with \eqref{eq:intersect_outside_S}, we have now proved \eqref{eq:vc_intersect_detail}, hence the theorem.
\end{proof}

For Theorem \ref{thm:vctoabc}, we will actually prove a slight generalization, which we state now:

\begin{thmbis}{thm:vctoabc}
Let $L_1, L_2, L_3$ be three lines of $\pp^2$ defined over $\overline \qq$ in general position.  We let $X_1$ be the blowup of $\pp^2$ at a point defined over $\overline \qq$ in $L_1 \setminus (L_2 \cup L_3)$ with $E_1$ as the exceptional divisor.  For $n\ge 2$, we construct $X_n$ inductively by blowing up $X_{n-1}$ at a point which is the intersection point of two of $\widetilde{L_1}, \widetilde{E_1}, \ldots, \widetilde{E_{n-2}}, E_{n-1}$, obtaining the new exceptional divisor $E_n$.  Further, let us assume that at least one blowup occurs at a point outside of $\widetilde{L_1}$. Then Vojta's inequality \eqref{eq:vojtacon} for such an $X_n$ with respect to the divisor
\[
\widetilde{L_1}+\widetilde{L_2}+\widetilde{L_3}+\widetilde{E_1} + \cdots +\widetilde{E_{n-1}} + E_n
\]
implies the inequality \eqref{ineq:abc} of the $abc$ conjecture with $1-3\rho - \epsilon$ in place of $1-\epsilon$ for subsets of the following special form: for a number field $k$, a finite set $S$ of places, $0\le \rho < \frac 13$, a constant $C$,
\begin{equation}\label{eq:primedthm}
\{(a,b,c): a\in \mathcal O_S,  |a|_S' \le C \cdot H(a)^\rho, \,\, b=1-a, \,\, c=1\}.
\end{equation}
Moreover, when $\rho$ is equal to $0$, the exceptional set to Vojta's inequality only affects finitely many $a$'s in \eqref{eq:primedthm}.
\end{thmbis}

The case of Theorem \ref{thm:vctoabc} is $\rho =0$ and $C=1$.  Thus, the $abc$ conjecture for $a\in \mathcal O_S^*$ and $c=1$ follows by adjusting the constant term of \eqref{ineq:abc}, deriving Theorem \ref{thm:vctoabc} from Theorem \ref{thm:vctoabc}$'$.

\begin{proof}[Proof of Theorem \ref{thm:vctoabc}$\,'$]
Suppose that after the blowup at $[1:0:1]$, all the blowups occur at intersection points of the strict transforms of $\Delta$ and previously-constructed exceptional divisors, and at least one blowup occurs at a point not on $\widetilde L$.  In general, there are many such intersection points, so each blowup may not necessarily occur at a point on the most-recently constructed exceptional divisor.  On the other hand, by working backwards from the first blowup that is not on $\widetilde{L}$ and picking up the relevant blowups, we can rearrange the order of the blowups (up to an isomorphism with the original space in question) so that the $i$-th blowup takes place at $\widetilde L \cap E_{i-1}$ for $2\le i <n$ and the $n$-th blowup takes place at a point in
\[
E_{n-1} \cap \left(\widetilde{E_1}\cup \cdots \cup \widetilde{E_{n-2}}\right) \quad \setminus \widetilde L,
\]
doing the rest of the blowups afterward.  In fact, by construction, $E_{n-1}$ only meets $\widetilde{E_{n-2}}$, so there is actually only one choice for the $n$-th blowup.  To repeat, we rearrange our blowups so that the second through $(n-1)$-th blowups occur at the intersection point of $\widetilde L$ with the most-recently-created exceptional divisor, the $n$-th blowup occurs at $E_{n-1} \cap \widetilde {E_{n-2}}$ (so not on $\widetilde L$), and the rest of the blowups occur at some intersection points of the strict transforms of $\Delta$ and various exceptional divisors from earlier stages. Since we know from \cite[Example 3.5.4]{vojta} that Vojta's conjecture on the blowup is stronger than  Vojta's conjecture on the base, it is sufficient to prove this theorem for this rearranged $n$-th blowup $X_n$.  We note that there is only one intersection point on the first blowup and it lies on $\widetilde L$, so $n$ must be at least $3$ to meet the hypothesis of the theorem.

According to Remark \ref{rem:blowupsbtree}, our situation corresponds to traversing the left fork at each stage to go down in the Stern--Brocot tree (Figure \ref{fig:sbtree}) until the very end when we take the right fork for the first time to level $n-1$.  Using Proposition \ref{prop:divandcoordring}, the first blowup point is defined by $x-1$ and $y$; the second blowup point is defined (locally) by $x-1$ and $\frac y{x-1}$.  Similarly, the $i$-th blowup point is defined locally by $x-1$ and $\frac y{(x-1)^{i-1}}$ for $i\le n-1$, and then the last blowup (the $n$-th one) is at a point defined locally by $\frac{(x-1)^{n-1}}y$ and $\frac y{(x-1)^{n-2}}$.

The description \eqref{eq:canon_div} of the canonical divisor $K_n$ on $X_n$ and the description \eqref{eq:div_intersect} of the normal-crossings divisor $D_n$  remain the same as in the proof of Theorem \ref{thm:vojtawithy}, so Vojta's inequality on $X_n$ with respect to $D_n$ is
\begin{align*}
\sum_{v\in S} &\lambda_v(D_n, P) + h_{X_n}(K_n, P)\\
&= - \sum_{v\notin S} \lambda_v(\Delta, P) + h_{X_1}(E_1, P) + \sum_{v\notin S} \sum_{i=2}^n \lambda_v(E_i, P)\,\,+\,\,O(1)\\
&< \epsilon h_{X_n}(A_n, P) + O(1),
\end{align*}
where $A_n$ is an ample divisor on $X_n$ and $O(1)$ indicates some constant (as we proceed further, the constant will change, but we will continue to simply denote by the big-$O$ notation).  Using the local equations of the blowup points given in the previous paragraph, we can use Proposition \ref{prop:blowupht} to rewrite this as
\begin{multline}\label{eq:vc_intersect_detail_abc}
\sum_{v\in M_k} \gcdpv(a-1, b)  + \sum_{i=2}^{n-1} \sum_{v\notin S} \gcdpv\left(a-1, \frac b{(a-1)^{i-1}}\right) \\
\phantom{AAAAAAAAA} +  \sum_{v\notin S} \gcdpv\left(\frac{(a-1)^{n-1}}b, \frac b{(a-1)^{n-2}}\right)\\
< \epsilon\max (h(a), h(b)) + \sum_{v\notin S} \lambda_v((XYZ=0), [a:b:1]) + O(1),
\end{multline}
where we also adjusted $\epsilon$ suitably in conjunction with the second conclusion of Proposition \ref{prop:ample}.

Note that it is enough to show the $abc$-type inequality \eqref{ineq:abc} for an enlarged $S$. In particular, we may assume that the set of $S$-integers is a PID.  Suppose that we fix a positive constant $C$, and we assume that $a\in \mathcal O_S$ satisfies $|a|_S' \le C\cdot H(a)^\rho$.
We factor $a-1$ as
\[
a-1 = \delta \cdot \prod_p p^{n_p},
\]
where $\delta \in \mathcal O_S^*$, $p$ is a generator of a prime ideal of $\mathcal O_S$, and $n_p$ is assumed to be at least $1$.  We now define $m_p$ by
\[
m_p = \begin{cases}
0 &\text{if } n_p = 1\\
\frac{n_p}2\cdot (2n-3) &\text{if } n_p \text{ is even}\\
\frac{n_p(2n-3) - 1}2 &\text{if } n_p \text{ is odd and at least } 3
\end{cases}
\]
and let $b = \prod_p p^{m_p}$.  We now analyze \eqref{eq:vc_intersect_detail_abc} for each $p$ with $n_p\ge 2$.  We will denote by $v(x) = v_p(x)$ the normalized additive valuation corresponding to $p$.  Let us first assume that $n_p$ is even.  The Farey interval of $\alpha_p = \frac{v_p(a-1)}{v_p(b)} = \frac{2}{2n-3}$ on the first level of the tree is $(0, 1)$, the second level of the tree is $(0, \frac 12)$, and so on, until the $(n-2)$-th level $(0, \frac 1{n-2})$ and then the Farey interval on the $(n-1)$-th level is $(\frac 1{n-1}, \frac 1{n-2})$; $\alpha_p$ itself appears for the first time in the tree at level $n$.  By noting
\begin{multline}\label{eq:locht_interval}
\gcdpv\left(\frac{(x-1)^{b_i}}{y^{a_i}}, \,\, \frac{y^{c_i}}{(x-1)^{d_i}}\right) \\
= v(y) \left[\max\left(0, \,\, \min\left(b_i \frac{v(x-1)}{v(y)} - a_i, \,\,\, c_i - d_i \frac{v(x-1)}{v(y)}\right)\right)\right],
\end{multline}
it follows immediately from Theorem \ref{thm:farey} that
\begin{align*}
v(p^{n_p}) + \sum_{i=2}^{n-1} &\gcdpv\left(a-1, \frac b{(a-1)^{i-1}}\right) +   \gcdpv\left(\frac{(a-1)^{n-1}}b, \frac b{(a-1)^{n-2}}\right) \\
&\qquad\qquad\qquad\qquad - \lambda_v((Y=0), [a:b:1])\\
&= v(p^{n_p}) + v(b) \cdot \left(1-\frac 1{\mathrm{denom}(\alpha_p)}\right) - v(b) = v(p^{n_p/2}).
\end{align*}
We next assume that $n_p$ is odd and at least $3$.  Since $\frac 2{2n-3} < \alpha_p = \frac {n_p}{m_p} < \frac 1{n-2}$, it follows that Farey intervals of $\alpha_p$ is the same as that for $\frac 2{2n-3}$ up to level $n-1$.  By Theorem \ref{thm:farey}, the function $\varphi\raisebox{-1.0ex}{}_{\!\!\scriptscriptstyle{\frac{2}{2n-3}}}$ evaluated at $\alpha_p$ must be
\begin{align*}
&\frac{\left(1-\frac 1{2n-3}\right) - \left(1-\frac 1{n-2}\right)}{\frac 2{2n-3} - \frac 1{n-2}} \cdot \left(\frac{n_p}{m_p} - \frac 2{2n-3}\right) + \left(1-\frac 1{2n-3}\right)\\
&= 1-\frac 1{2n-3} - \frac {n-1}{m_p (2n-3)}.
\end{align*}
Therefore, \eqref{eq:locht_interval} shows that
\begin{align*}
v(p^{n_p}) + \sum_{i=2}^{n-1} &\gcdpv\left(a-1, \frac b{(a-1)^{i-1}}\right) +   \gcdpv\left(\frac{(a-1)^{n-1}}b, \frac b{(a-1)^{n-2}}\right) \\
&\qquad\qquad\qquad\qquad - \lambda_v((Y=0), [a:b:1])\\
&= v(p^{n_p}) + v(b) \cdot \left(1-\frac 1{2n-3} -\frac {n-1}{m_p (2n-3)}\right) - v(b)\\
&= v(p^{n_p}) + v(p) \cdot \left(- \frac{m_p}{2n-3} - \frac{n-1}{2n-3}\right)\\
&= v(p^{n_p}) + v(p) \cdot \left(-\frac{n_p}2 + \frac 1{2(2n-3)} - \frac{n-1}{2n-3}\right) = v(p)\cdot \frac{n_p-1}2.
\end{align*}
Ignoring the contributions from $\gcdpv(a-1, b)$ for $v\in S$, we now see that \eqref{eq:vc_intersect_detail_abc} implies
\begin{equation}\label{eq:onepoweroff}
\sum_{p: n_p \text{ even}} v_p(p^{n_p/2}) + \sum_{\substack{p: n_p \text{ odd}\\n_p\ge 3}} v_p(p^{(n_p-1)/2}) < \epsilon \max(h(a), h(b)) + \log |a|_S' + O(1).
\end{equation}
By our construction of $b$, we have $h(b) \le n(h(a-1)) \le n h(a) + n \log 2$.  Therefore, Proposition \ref{prop:ridout} and \eqref{eq:onepoweroff} now implies that
\begin{equation}\label{eq:vojtatoabc_final}
\sum_{v\in S} \max(0, v(a-1)) + \sum_{v_p\notin S} v_p(p^{n_p}) - \sum_{\substack{v_p\notin S\\n_p \text{ odd}}} v_p(p) < (1+2n)\epsilon h(a) + 3 \log |a|_S' + O(1).
\end{equation}
By subtracting \eqref{eq:vojtatoabc_final} from $h(a-1) = h(a)+O(1)$, we conclude that
\[
\sum_{\substack{v_p\notin S\\n_p \text{ odd}}} v_p(p)   > (1-(1+2n)\epsilon) h(a) - 3 \log |a|_S'  - O(1) \ge (1-(1+2n)\epsilon - 3\rho) h(a) - O(1)
\]
by our hypothesis $|a|_S' \le C\cdot H(a)^\rho$.  Since the truncated local height function $\trunloc((1), a)$ is at least the left-hand side, we have now derived the weakened inequality of the $abc$ conjecture, with $1-3\rho - \epsilon'$ in place of $1-\epsilon'$.

So far, what we have shown is that the inequality in Vojta's conjecture implies the weakened inequality of the $abc$ conjecture.  Since Vojta's conjecture itself has some exceptional set where the inequality might not hold, we have to make sure that the points we plug into \eqref{eq:vc_intersect_detail_abc} do not lie in the exceptional set.  For this part, we need to further assume that $\rho = 0$.  Since only finitely many rational primes $q$ satisfy $\log q \le C^{[k:\qq]}$, we can enlarge $S$ to $S'$ so that any $a$ satisfying $|a|_S' \le C$ is actually an $S'$-unit. Pulling back the points $[a:b:1]$ as described above by the (rational) algebraic map $(x,y)\mapsto (x, y^{2n-3})$, the points satisfy $y^2 | (x-1)$ with $x$ being a unit and $y$ being integral.  The proof of \cite[Proposition 2]{yasufuku2} shows that an irreducible curve containing infinitely many such points cannot be in an exceptional set.  Since the exceptional set of Vojta's inequality is a finite union of algebraic curves, it now follows that only finitely many $a$'s are affected.
\end{proof}

\begin{remark}\label{rem:vctoabc_ratl}
We can apply the same tactics for a non-integral $a = \frac AB$ with $\mathrm{gcd}(A,B) = 1$.  This enables us to obtain the same $abc$-type inequality for the relation $A + (-B) = A- B$ when $|AB|_S' \le C\cdot H(a)^\rho$.  In this case, we factor $A-B$ instead of $a-1$ and define $m_p$ and $b$ in a similar fashion, evaluating \eqref{eq:vc_intersect_detail_abc}
at $[A/B : b : 1]$.  Since $\lambda_v((Z=0), [a:b:1])$ is $v(B)$ and $\sum_{v\notin S} |v(a)| = \sum_{v\notin S} v(AB)$, \eqref{eq:onepoweroff} and \eqref{eq:vojtatoabc_final} hold with $\log |AB|_S'$ in place of $\log |a|_S'$.  The rest of the argument proceeds as before.  We will make a similar argument for non-integral elements in the proof of Theorem \ref{thm:abctovc}, so we skip the details here.

Again, this only shows that the inequality of Vojta's conjecture implies the weakened inequality of the $abc$ conjecture with $1-3\rho - \epsilon'$. To deal with exceptional sets, one would need a different argument for non-$S$-units.
\end{remark}

\begin{proof}[Proof of Theorem \ref{thm:abctovc}]
Vojta's conjecture is stronger for a larger $S$, so we may again assume that $\mathcal O_S$ is a PID.  Since we can remove finitely many curves from consideration, it will be sufficient to prove Vojta's inequality for points of the form $[a:b:1]$ with $a\in k\setminus \{0,1\}$ and $b\in k^*$.  As before, the canonical divisor on $X_n$ can be described by \eqref{eq:canon_div} and the normal-crossings divisor $D_n$ is given by \eqref{eq:div_intersect}, so the goal is again to prove
\begin{equation}\label{eq:vc_intersect2}
h_{X_1}(E_1, P) + \sum_{v\notin S} \sum_{i=2}^n \lambda_v(E_i, P)
< \epsilon h_{X_n}(A_n, P) + \sum_{v\notin S} \lambda_v(\Delta, P) + O(1),
\end{equation}
where $P = [a:b:1]$ and $A_n$ is an ample divisor on $X_n$.  Since there are always two points on the newest exceptional divisor which meet with other components of $D_n$, it is possible to go back to the one that was not blown up initially.  That is, the blowup might occur at a point not on the most-recently-created exceptional divisor.

We first assume that $a\in \mathcal O_S$, and let
\[
a-1 = \delta \cdot \prod_{v_p\notin S} p^{n_p},
\]
where $\delta \in \mathcal O_S^*$, $p$ is a generator of a prime ideal in $\mathcal O_S$, and $n_p >0$.

The general strategy of the proof is to analyze \eqref{eq:vc_intersect2} for each $v\notin S$, eventually obtaining \eqref{eq:abctovc_vadic} below.  We then sum this over $M_k\setminus S$ and combine it with the global information about the largeness of the radical of $a-1$, provided to us by the $abc$ conjecture. This leads us to Vojta's conjecture in this setting.  For the local analysis leading to \eqref{eq:abctovc_vadic}, we rearrange the order of the blowups so that all of the $v=v_p$-adic local height contributions come in the first $m$ blowups.  This corresponds to a downward path in the Stern--Brocot tree (Figure \ref{fig:sbtree}) which keeps $\alpha_p = \frac{v_p(a-1)}{v_p(b)}$ inside the Farey interval up to level $m-1$. Now, the sum of the $v$-adic local height contributions will be the largest for the point $[a:b:1]$ if we go down in the Stern--Brocot tree all the way to $\frac{v_p(a-1)}{v_p(b)}$ (note that this is a rational number), and the contributions in this case are computed using Theorem \ref{thm:farey}.  Our sequence of $m$ blowups is shorter than the longest path, so we obtain an upper bound for the sum of the $v$-adic contributions.  To summarize, we will be rearranging the blowups in a different way for each $v$ to get an estimate of the $v$-adic contributions, but the final estimate \eqref{eq:abctovc_vadic} does not involve the rearrangements, so we can combine this information together and compare with the global information from the $abc$ conjecture.

We now make this precise.  For each $p$, if $v_p(b) \le 0$ then there is no $v$-adic contribution at all, and if $v_p(b) \ge v_p(a-1)$, then the only $v$-adic contribution comes from the first exceptional divisor.  Otherwise, we will rearrange the order of blowups using the notion of \textit{relevant subtowers} \cite{yasufuku}.  More specifically, for each $E_i$ in the original ordering, we can find its \textbf{relevant subset} of $\{1,\ldots, i\}$ by picking out the blowups necessary to construct $E_i$.  That is, the subset is constructed by looking at the image of $E_i$ to $X_{i-1}$ and determining the largest index $j$ for which this point is on $\widetilde{E_j}$, and working inductively backward to the first blowup.  By just going through the relevant subset of $E_i$, we can construct the blowup sequence to $E_i$, called the \textbf{relevant subtower}.  Now, by construction, every relevant subtower has the property that each blowup occurs on the most recently-constructed exceptional divisor.  Since we are only doing the intersection blowups, this subtower corresponds to traversing the Stern--Brocot tree, as discussed in Proposition \ref{prop:divandcoordring} and Remark \ref{rem:blowupsbtree}.  Among the relevant subtowers of $E_1,\ldots, E_n$, we look at the maximal (in terms of level) subtower for which $\alpha_p = \frac{v_p(a-1)}{v_p(b)}$ stays inside the Farey interval for all levels.  Let us assume that this maximal subtower has $m-1$ levels; since the first level of the tree is already the second blowup, this means that it requires $m$ blowups to construct this subtower.  By \cite[Proposition 11]{yasufuku}, we may reorder our blowups (up to an isomorphism with the original space) so that this subtower of $m$ blowups is constructed first and then the rest of the remaining blowups later.

After this rearrangement, the relevant subtower for the $(i+1)$-th exceptional divisor for $i\ge m$ corresponds to going down the Stern--Brocot tree where the final Farey interval does not contain $\frac{v_p(a-1)}{v_p(b)}$. By Proposition \ref{prop:divandcoordring} and \eqref{eq:locht_interval}, it follows immediately that there is no $v$-adic local height contribution from such an exceptional divisor.  Therefore,
we may now assume that we are traversing the Stern--Brocot tree up to $m-1$ levels, and the corresponding $m$ exceptional divisors will be the ones contributing to the $v$-adic part of \eqref{eq:vc_intersect2} at $P = [a:b:1]$.
By Propositions \ref{prop:blowupht} and \ref{prop:divandcoordring}, the $v=v_p$-adic part of the left-hand side of \eqref{eq:vc_intersect2} is
\begin{equation}\label{eq:vc_intersect_detail_abc2}
\gcdpv(a-1, b)  + \sum_{i=1}^{m-1} \gcdpv\left(\frac{(a-1)^{b_i(\alpha_p)}}{b^{a_i(\alpha_p)}}, \frac {b^{c_i(\alpha_p)}}{(a-1)^{d_i(\alpha_p)}}\right).
\end{equation}
It might actually take more than $m-1$ levels to traverse the Stern--Brocot tree all the way down to $\alpha_p$; by a similar argument as done in the proof of Theorem \ref{thm:vctoabc}$'$, Theorem \ref{thm:farey} and \eqref{eq:locht_interval} show that the contribution in this longest case is
\[
\gcdpv(a-1, b) + v(b) \cdot \frac{\mathrm{denom}(\alpha_p)-1}{\mathrm{denom}(\alpha_p)}  \le \gcdpv(a-1, b) + v(b/p).
\]
Our subtower of $m$ blowups is stopping in the middle of this path all the way to $\alpha_p$, so \eqref{eq:vc_intersect_detail_abc2} will be no bigger.  Therefore,  the $v$-adic part of the left-hand side of \eqref{eq:vc_intersect2} satisfies
\begin{align}
\sum_{i=1}^n \lambda_v(E_i, P) &\le \gcdpv(a-1, b) + v(b/p) \le v(p^{n_p}) + v(b/p)\notag \\
&= v(p^{n_p-1}) + v(b) = v(p^{n_p-1}) + \lambda_v((Y=0),P).  \label{eq:abctovc_vadic}
\end{align}
Note that this is valid for all $v=v_p\notin S$ for which $n_p>0$, even when $v_p(b)\le 0$ or $v_p(b) \ge v_p(a-1)$.  Now, we switch from the $v=v_p$-adic analysis to a global analysis.  From the $abc$ conjecture (Conjecture \ref{conj:abc}), our assumption that $a\in \mathcal O_S$ implies
\[
\sum_{\substack{v_p\notin S\\n_p>0}} v_p(p) =  \sum_{v\notin S} \trunloc((1),a) \ge (1-\epsilon) h(a) - \sum_{v\notin S} \trunloc((0),a) - O(1).
\]
Subtracting this from
\begin{align*}
\sum_{v\in S} \vplus(a-1) &+ \sum_{v_p\notin S} v_p(p^{n_p}) = \sum_{v\in M_k} \max(-\log |a-1|_v, 0)\\
&= h\left(\frac 1{a-1}\right) = h(a-1) = h(a)+O(1),
\end{align*}
we obtain
\begin{equation}\label{eq:abctovc_spart}
\sum_{v\in S} \vplus(a-1) + \sum_{\substack{v_p\notin S\\n_p>0}} v_p(p^{n_p-1}) \le \epsilon h(a) + \sum_{v\notin S} \trunloc((0),a) + O(1).
\end{equation}
Combining with \eqref{eq:abctovc_vadic}, we obtain
\begin{align*}
h_{X_1}(E_1, P) &+ \sum_{i=2}^n \sum_{v\notin S} \lambda_v(E_i, P) = \sum_{v\in S} \gcdpv(a-1, b) + \sum_{i=1}^n \sum_{v\notin S} \lambda_v(E_i, P)\\
&\le \sum_{v\in S} \vplus(a-1) + \sum_{\substack{v_p\notin S\\n_p>0}} \left(v_p(p^{n_p-1}) + \lambda_{v_p}((Y=0),P)\right)\\
&\le \epsilon h(a) + \sum_{v\notin S} \trunloc((0),a) + \sum_{v\notin S} \lambda_v((Y=0),P) + O(1)\\
&\le \epsilon h(a) + \sum_{v\notin S} \lambda_v((XY=0), P).
\end{align*}
Using Proposition \ref{prop:ample}, we have now derived \eqref{eq:vc_intersect2}, concluding the proof of the theorem when $a\in \mathcal O_S$.

We now assume that $a = \frac AB$, $A, B\in \mathcal O_S$ with $\mathrm{gcd}(A,B) = 1$.  Then writing $A - B = \delta \cdot \prod p^{n_p}$ as before, the $abc$ conjecture for $a$ says that
\[
\sum_{\substack{v_p\notin S\\n_p>0}} v_p(p)  \ge (1-\epsilon) h(a) - \sum_{v\notin S} \left(\trunloc((0), A) + \trunloc ((0), B)\right) - O(1).
\]
Note that the denominators of $a$ and $b$ do not contribute to $\lambda_v(E_i, P)$ for $v\notin S$, since $E_i$ is inside $\pi^{-1}(E_1)$.
Therefore, \eqref{eq:abctovc_vadic} remains valid, and the same argument as above yields
\begin{align*}
h_{X_1}&(E_1, P) + \sum_{i=2}^n \sum_{v\notin S} \lambda_v(E_i, P) \\
&\le \epsilon h(a) + \sum_{v\notin S} \left(\trunloc((0),A) + \trunloc((0), B)\right) + \sum_{v\notin S} \lambda_v((Y=0),P) + O(1)\\
&\le \epsilon h(a) + \sum_{v\notin S} \lambda_v((XYZ=0), P) + O(1),
\end{align*}
as $v(A) \le \lambda_v((X=0),[a:b:1])$ and $v(B) \le \lambda_v((Z=0),[a:b:1])$.  We have now proved \eqref{eq:vc_intersect2} in all cases, concluding the proof of Theorem \ref{thm:abctovc}.
\end{proof}

We end this paper by extending Theorems \ref{thm:vojtawithy} and \ref{thm:abctovc} to the setting where we blow up multiple points on one of the lines in $\pp^2$.

\begin{proposition}\label{prop:extend}
Let $L_1, L_2, L_3$ be three lines of $\pp^2$ defined over $\overline \qq$ in general position.  Let $p_1, \ldots, p_\ell$ be distinct points on $L_1\setminus (L_2\cup L_3)$ defined over $\overline \qq$.  Let $X'$ be the blowup of $\pp^2$ at all of $p_1,\ldots, p_\ell$.  We then construct $X$ and $\pi: X\too \pp^2$ from $X'$ by a sequence of blowups, where each blowup occurs at an intersection of the irreducible divisors which are in the inverse image of $L_1$.
\begin{itemize}
\item[(i)] If each blowup occurs at a point on $\widetilde{L_1}$, then Vojta's conjecture for the reduced part of $\pi^*(L_1 + L_2 + L_3)$ holds on $X$.

\item[(ii)] Otherwise, the $abc$ conjecture implies Vojta's conjecture for the reduced part of $\pi^*(L_1 + L_2 + L_3)$ on $X$.
\end{itemize}
\end{proposition}

\begin{proof}
We note that the canonical divisor formula \eqref{eq:canon_div} for $X$ remains the same.  Moreover, the pullback of $L_1 + L_2 + L_3$ to $X'$ is a normal-crossings divisor containing each of the exceptional divisor over $p_i$, so the reduced part of $\pi^*(L_1 + L_2 + L_3)$ is still given by \eqref{eq:div_intersect}, by the same argument as in the proof of Proposition \ref{prop:divandcoordring}.  Now, by a coordinate change on $\pp^2$, we may assume that $L_1 = (Y=0)$, $L_2$ and $L_3$ are defined by $XZ = 0$, and $p_i = [a_i : 0: 1]$ with $a_i\in k^*$.  As usual, it suffices to prove Vojta's conjecture for a larger $S$, so we may assume that $\mathcal O_S$ is a PID and $a_i - a_{i'}$ are all $S$-units for all $i\neq i'$.  As in the proof of Theorem \ref{thm:abctovc}, for each $v\notin S$, there is a maximal subtower which contributes all of the local heights over $p_i$, and this subtower corresponds to a path in the Stern--Brocot tree.  Therefore, the left-hand-side of \eqref{eq:vc_intersect_detail} (or \eqref{eq:vc_intersect2}) in our case is
\begin{equation}\label{eq:multiplepts}
\sum_{i=1}^\ell \sum_{v\in M_k} \gcdpv(a-a_i, b) + \sum_{i=1}^{\ell} \sum_{v\notin S}  \sum_j \gcdpv\left(\frac{(a-a_i)^{b_{i,v,j}}}{b^{a_{i,v,j}}}, \frac {b^{c_{i,v,j}}}{(a-a_i)^{d_{i,v,j}}}\right),
\end{equation}
where the inner index $j$ goes through some subtower over $p_i$ for each $i$ and $v$ and $\frac{a_{i,v,j}}{b_{i,v,j}}$ and $\frac{c_{i,v,j}}{d_{i,v,j}}$ are Farey neighbors.  By our assumption on $S$, for each $v\notin S$, only one $i$ contributes positively to \eqref{eq:multiplepts}, so the $v$-adic computation in \eqref{eq:multiplepts} for $v\notin S$ goes by the same argument as in the proofs of Theorems \ref{thm:vojtawithy} and \ref{thm:abctovc}.

For the $S$-part of \eqref{eq:multiplepts}, we argue differently for cases (i) and (ii).  For case (i), we follow the ideas of \cite[Theorem 6]{yasufuku2}, which proved Vojta's conjecture on $X'$.  For each $v\in S$, we can choose a sufficiently small radius so that the balls centered at $a_i$'s do not overlap. Thus, there exists some constant $C_v$ so that
\[
\sum_{i=1}^\ell \vplus(a-a_i) \le \max_i(\vplus(a-a_i)) + C_v.
\]
We now have one linear expression for each $v$ instead of a sum of $\ell$ expressions, so we can apply Schmidt's subspace theorem (Theorem \ref{thm:schmidt}) in the way we proved Proposition \ref{prop:ridout}; in fact, we apply the argument $\ell^{|S|}$ times, depending on which $i$ gives the maximum $\vplus(a-a_i)$ for each $v\in S$.  This enables us to prove Proposition \ref{prop:ridout} with the left-hand-side replaced by
\[
\sum_{i=1}^\ell \sum_{v\in S} \vplus(a-a_i),
\]
so now the rest of the proof of Theorem \ref{thm:vojtawithy} goes through.

For case (ii), we first remark that the $abc$ conjecture is equivalent to the following form, through the use of Belyi maps (cf. \cite{vF}):
\[
\sum_{i=1}^\ell \sum_{v\notin S} \trunloc((a_i),a) \ge (\ell-\epsilon) h(a) - \sum_{v\notin S} \trunloc((0), a) - O(1) \qquad \forall a\in\mathcal O_S.
\]
We note that we are really looking at $(0), (\infty), (a_1), \ldots, (a_\ell)$, so we have a total of $\ell+2$ points.  We now subtract this from $\sum_{i=1}^\ell h(a-a_i) = \ell h(a) + O(1)$ to obtain the following analog of \eqref{eq:abctovc_spart}:
\[
\sum_{i=1}^\ell \sum_{v\in S} \vplus(a-a_i) + \sum_{i=1}^\ell \sum_{\substack{v_p\notin S\\n_{i,p}>0}} v_p(p^{n_{i,p}-1}) \le \epsilon h(a) + \sum_{v\notin S} \trunloc((0),a) + O(1).
\]
The rest of the proof is the same as the proof of Theorem \ref{thm:abctovc}.
\end{proof}

\bibliographystyle{amsplain}
\bibliography{nevan}

\vspace{0.2in} \footnotesize
              \textsc{Department of Mathematics, College of Science and Technology, Nihon University, 1-8-14 Kanda-Surugadai Chiyoda, Tokyo 101-8308, JAPAN}\\

\vspace{0.1in} \texttt{yasufuku@math.cst.nihon-u.ac.jp}\\
\end{document}